\documentclass[11pt,a4paper]{amsart}
\usepackage{amssymb,amsmath}
\usepackage{graphicx}
\numberwithin{equation}{section}
\vspace{8cm}
\newtheorem{theorem}{Theorem}[section]

\newtheorem{corollary}[theorem]{Corollary}
\newtheorem{lemma}[theorem]{Lemma}

\newtheorem{remark}[theorem]{Remark}

\newcommand{\R}{{\mathbb R}}
\newcommand{\Rn}{{\mathbb R}^n}

\newcommand{\vro}{\varrho}

\newcommand{\vfi}{\varphi}

\def\d0{\delta_0}

\def\sc0{(x_0,y_0)}
\newcommand{\beq}{\begin{equation}}
\newcommand{\eeq}{\end{equation}}
\newcommand{\refe}[1]{{\rm (\ref{#1})}}
\newcommand{\dis}{\displaystyle}

\newcommand{\Mm}{\mathcal{M}^-_{\lambda, \Lambda}}
\newcommand{\Mp}{\mathcal{M}^+_{\lambda, \Lambda}}
\newcommand{\Ll}{\frac{\Lambda}{\lambda}}

\begin{document}
\parindent=0pt

\title[Homogeneous solutions of extremal Pucci's equations]{Homogeneous solutions of extremal Pucci's equations in  planar cones}

\author[ F. Leoni]
{ Fabiana Leoni}

\address{Dipartimento di Matematica\newline
\indent Sapienza Universit\`a  di Roma \newline
 \indent   P.le Aldo  Moro 2, I--00185 Roma, Italy.}
\email{leoni@mat.uniroma1.it}

\begin{abstract}
We derive explicit expressions of the homogeneous solutions  in two dimensional cones  for Pucci's extremal equations. As  examples of possible applications, we obtain monotonicity formulas for all nonnegative supersolutions and necessary and sufficient  explicit conditions for non--existence results of Liouville type.
 \end{abstract}
\maketitle

\section{Introduction}\label{intro}
The goal of the present paper is to determine explicit solutions in two dimensional cones of homogeneous Dirichlet boundary value problems associated with the extremal elliptic equations
\begin{equation}\label{general}
{\mathcal M}^\pm_{\lambda, \Lambda}(D^2u)=0\, ,
\end{equation}
where $\Mp$ and $\Mm$ are the Pucci's extremal operators introduced by C. Pucci in \cite{P}.

We recall that, given two ellipticity constants $\Lambda\geq \lambda>0$, the operator $\Mp$, as a function of $M\in {\mathcal S}_n$ where ${\mathcal S}_n$ is the set of $n\times n$ symmetric matrices, can be equivalently defined either as
$$
\Mp (M) =\Lambda\, \sum_{\mu_i>0} \mu_i + \lambda\, \sum_{\mu_i<0} \mu_i
$$
where $\mu_1, \ldots ,\mu_n$ are the eigenvalues of $M$, or as
$$
\Mp (M) =\sup_{A\in {\mathcal A}_{\lambda, \Lambda}} {\rm tr} \left( A M\right)
$$
with ${\mathcal A}_{\lambda, \Lambda}=\{ A\in {\mathcal S}_n\, :\, \lambda\, I_n\leq A\leq \Lambda\, I_n\}$, where $I_n$ is the identity matrix in ${\mathcal S}_n$ and $\leq $ is the usual partial ordering in ${\mathcal S}_n$, meaning that $A\leq B$ if and only if $A-B$ has nonpositive eigenvalues.

The operator $\Mm$ is likewise defined either as
$$
\Mm (M) =\lambda\, \sum_{\mu_i>0} \mu_i + \Lambda\, \sum_{\mu_i<0} \mu_i
$$
or as
$$
\Mm (M) =\inf_{A\in {\mathcal A}_{\lambda, \Lambda}} {\rm tr} \left( A M\right)\, ,
$$
and it is easy to verify that
\begin{equation}\label{Mpm}
\Mp (M)= -\Mm(-M)\quad \hbox{for all } M\in {\mathcal S}_n\, .
\end{equation}
The Pucci operators ${\mathcal M}^\pm_{\lambda, \Lambda}$ are extremal not only in the class of linear elliptic operators, but they satisfy
$$
\Mm (M) \leq F(x,M)\leq \Mp (M) \quad \hbox{for all } M\in {\mathcal S}_n
$$
for any uniformly elliptic opearator $F$ having ellipticity constants $\Lambda$ and $\lambda$ and satisfying $F(x,O)=0$. This explains the central role played by ${\mathcal M}^\pm_{\lambda, \Lambda}$ in the elliptic theory, see the monograph \cite{CC}, and the importance of finding explicit solutions of \refe{general} which act as barrier functions in comparison theorems for solutions of more general fully nonlinear equations.

In particular, solutions of \refe{general} in cone--like domains may be used in the asymptotic analysis near corners of solutions in Lipschitz domains of general elliptic equations with asymptotically negligible lower order terms, as performed for divergence form operators, see e.g. \cite{D}.

Equations \refe{general} are known to have homogeneous solutions in cone--like domains of $\R^n$ for any dimension $n\geq 2$, see \cite{ASS, M, O}. In particular, the more recent results of \cite{ASS} establish that, for any positively homogeneous and uniformly elliptic operator $F$ and any cone ${\mathcal C}\subset \R^n$, the homogeneous Dirichlet problem
\begin{equation}\label{PF}
\left\{
\begin{array}{cl}
F(D^2u)=0 &  \hbox{ in } {\mathcal C}\\
u=0 &  \hbox{ on } \partial {\mathcal C}\setminus \{0\}
\end{array}
\right.
\end{equation}
has exactly four (up to normalization) constant sign homogeneous solutions, that are of the form
$$
 \begin{array}{ll}
 u_{\alpha^{\pm}}(x)= |x|^{\alpha^\pm} \phi_{\alpha^\pm}\left( \frac{x}{|x|}\right) & \hbox{with } \alpha^-<0<\alpha^+,\ \phi_{\alpha^\pm}>0\\[2ex]
  v_{\beta^{\pm}}(x)= |x|^{\beta^\pm} \psi_{\beta^\pm}\left( \frac{x}{|x|}\right) & \hbox{with } \beta^-<0<\beta^+,\ \psi_{\beta^\pm}<0
  \end{array}
  $$
 and  $u_{\alpha^{\pm}},\ v_{\beta^\pm}$ are the only constant sign solutions of \refe{PF} which are bounded either for $|x|\to 0$ or as $|x|\to +\infty$.
The homogeneity exponents $\alpha^\pm$ and $\beta^\pm$ are defined in \cite{ASS} by means of inf or sup formulas which resemble the maximum principle based approach to the theory of the principal eigenvalues, and  their evaluation requires the determination of explicit sub or supersolutions of problem \refe{PF}. When the operator $F$ is one of the Pucci extremal operators and the cone ${\mathcal C}$ is axially symmetric, a direct approach for finding the homogeneous solutions of \refe{PF} has been proposed in \cite{M}, where the one variable functions $\phi_{\alpha^\pm}$ and $\psi_{\beta^\pm}$ are proved to solve  fully nonlinear eigenvalue problems for ODE. The existence of  the pairs $(\alpha^\pm,\phi_{\alpha^\pm})$ and 
$(\beta^\pm,\psi_{\beta^\pm})$ is obtained by means of well-posedness  results for ODEs, and also in this case no explicit evaluation is provided. In the two dimensional case $n=2$, the autonomous second order ODE problems for  $(\alpha^\pm,\phi_{\alpha^\pm})$ and 
$(\beta^\pm,\psi_{\beta^\pm})$ have been reduced in \cite{O} to first order systems, which have been integrated by introducing polar coohrdinates in the phase plane. This method leads to implicit expressions both for the homogeneity exponents and for the angular functions.

In this paper we consider as in \cite{O} the two dimensional problem for the Pucci operators, and we integrate directly the fully nonlinear ODE obtained for the angular functions $\phi_\alpha^\pm$ and $\psi_\beta^\pm$. This leads to an explicit formula relating the homogeneity exponents $\alpha^\pm$ and $\beta^\pm$ with the angle measuring the amplitude of the cone and the ellipticity constants of the operators, see Theorem \ref{homS}.

Several consequences may be deduced from the explicit expression of the homogeneity exponents. As an example, it turns out that for the operator $\Mp$ positive homogeneous order two solutions  occur only in cones of amplitude $2\arctan \sqrt{\frac{\lambda}{\Lambda}}$, as well as negative
homogeneous order two solutions  occur only in cones of amplitude $2\arctan \sqrt{\frac{\Lambda}{\lambda}}$. 
This implies that, when looking for principal eigenfunctions of $\Mp$ as functions of separable variables in  planar domains analogous to rectangles for Laplace operators, one has to consider Lipschitz domains with corners of amplitude $2\arctan \sqrt{\frac{\lambda}{\Lambda}}$. This explicit construction has been performed in \cite{BL}.
Similarly, if a "second" eigenfunction for the operator $\Mp$ in the unit disk does exist, as discussed in \cite{BLP}, it must have a nodal line intersecting the boundary of the unit disk with a corner of amplitude $2\arctan \sqrt{\frac{\lambda}{\Lambda}}$.

The homogeneous solutions of equations \refe{general} in cone--like domains allow to prove extended comparison theorems of Phragmen--Lindel\"of type, H\"older regularity results and removable boundary singularity results for fully nonlinear elliptic equations, as proved in \cite{ASS, M}. Moreover, the homogeneity exponents $\alpha^\pm$ and $\beta^\pm$ are intimately related to the critical exponents which separate the zones of existence from those of non existence for the exponent $p\in \R$ of the problems
$$
u>0\, ,\quad {\mathcal M}^\pm (D^u)+u^p\leq 0\quad \hbox{ in  } {\mathcal C}\, .
$$
As discussed in \cite{AS,L}, if we consider for instance the operator $\Mm$ in the above differential inequality, then no positive solution $u$ exists provided that
\begin{equation}\label{palfa}
1-\frac{1}{\alpha^+}\leq p\leq 1-\frac{1}{\alpha^-}\, .
\end{equation}
By using the explicit expressions of the functions $\phi_{\alpha^\pm}$,  we can actually prove, at least in the case of two dimensional cones, that the above restrictions on $p$ are sharp, by exhibiting explicit solutions in the cases $p<1-\frac{1}{\alpha^+}$ and $p>1-\frac{1}{\alpha^-}$, see Theorem \ref{Liou}.

As further possible applications of homogeneous solutions, let us finally mention that they have been used for fully nonlinear equations also as counterexamples in regularity theory \cite{NTV}, or as barriers for studying overdetermined problems \cite{BD}.

The paper is organized as follows: in Section 2 we construct the homogeneous solutions and derive several relationships between the homogeneity exponents; in Section 3, after showing some simple monotonicity formulas for all positive supersolutions of uniformly elliptic equations, we prove the optimality of conditions \refe{palfa}



\section{Homogeneous solutions}

This section is devoted to determine explicitly   the  homogeneous solutions of equations \refe{general} in planar cones. Because of  identity \refe{Mpm}, it is enough to consider the operator $\Mm$ and to construct its positive and negative homogeneous solutions. Thus, we consider the homogeneous problem
\begin{equation}\label{homP}
\left\{ \begin{array}{c}
\Mm (D^2u) =0 \ \hbox{ in } \mathcal{C}_0\, ,\\[2ex]
u=0 \ \hbox{ on } \partial \mathcal{C}_0\setminus \{0\}\, ,
\end{array} \right.
\end{equation}
where $\mathcal{C}_0$ is the planar cone
$$
\mathcal{C}_0=\{ (x,y)\in \R^2\, :\, y> \cos (\theta_0) \sqrt{x^2+y^2} \}
$$
and $\theta_0\in (0, \pi)$ is the half-opening of $\mathcal{C}_0$.

We will make  use of the explicit computation of the eigenvalues of the hessian matrix for symmetric homogeneous functions which hold true in  any dimension $n\geq 2$. Let us start with the following algebraic  result, whose proof is easily verified by straightforward computation.

\begin{lemma}\label{eigenvalues}
Let $v,\ w\in \Rn$ be unitary vectors and, given $a,  b, c, d\in \R$, let us consider the symmetric matrix
$$
A = a\,  v\otimes v + b \, w\otimes w + c \, (v\otimes w +w\otimes v) +d\, I_n\, ,
$$
where $v\otimes w$ denotes the $n\times n$ matrix whose $i, j$-entry is $v_i w_j$. Then, the eigenvalues of $A$ are:
\begin{itemize}

\item $d$, with multiplicity (at least) $n-2$ and eigenspace given by $<v, w>^{\bot}$;
\smallskip

\item $\displaystyle  d+\frac{a+b+2cv\cdot w\pm \sqrt{(a+b+2cv\cdot w)^2+4(1-(v\cdot w)^2)(c^2-ab)}}{2}$, 
\newline which are simple (if different from $d$).
\end{itemize}
In particular, if either $c^2=ab$ or $(v\cdot w)^2=1$, then the eigenvalues are $d$, which has multiplicity $n-1$,  and $d+a+b+2cv\cdot w$,
  which is simple.
\end{lemma}
\smallskip

In the sequel we will use the notation, for $x\in \R^n$, $x=(x', x_n)$ with $x'\in \R^{n-1}$ and $x_n\in \R$.

\begin{lemma}\label{hessian}
Let $\Phi :\Rn \setminus \{ (0,x_n)\, :\, x_n\leq 0\} \to \R$ be a $C^2$  symmetric homogeneous function of the form
$$
\Phi (x) = \vro ^\alpha \phi (\theta)\, ,
$$
with $\vro =|x|\,,\ \theta =\arccos \left( \frac{x_n}{|x|}\right)\, ,\ \alpha \in \R$, for $\phi :[0,\pi )\to \R$ of class $C^2$ and satisfying $\phi'(0)=0$.
Then, the eigenvalues of the hessian matrix $D^2\Phi (x)$ are
\begin{itemize}

\item $\dis \vro^{\alpha-2} \left( \alpha\,  \phi +\frac{ \phi'}{\tan \theta} \right)$, with multiplicity  $n-2$ ; 
\smallskip

\item $\displaystyle  \frac{\vro^{\alpha-2}}{2} \left( \alpha^2 \phi +\phi'' \pm \sqrt{ \left[ \alpha (\alpha -2) \phi -\phi''\right]^2 +4 (\alpha -1)^2 \phi'^2} \right)$,  which are simple.
\end{itemize}
\end{lemma}

\begin{proof} Let $\vfi :(-1,1] \to \R$ be a function of class $C^2$ and let us consider functions $\Phi: \Rn\setminus \{(0,x_n)\, :\, x_n\leq 0\} \to \R$ of the form 
$$
\Phi (x)= \vro^\alpha \vfi \left( \frac{x_n}{\vro}\right)\, .
$$
Then, a direct computation shows that
$$
\begin{array}{rl}
\dis D^2\Phi (x) = & \dis \vro^{\alpha -2} \left\{ \left( \alpha\,  \vfi -t \vfi'\right) I_n + \left( \alpha (\alpha -2) \vfi +(3-2\alpha) t \vfi' +t^2\vfi''\right) \frac{x}{\vro}\otimes \frac{x}{\vro} \right.\\[2ex]
& \dis \quad \left. +  \vfi'' {\bf e}_n\otimes {\bf e}_n  + \left( (\alpha -1)\vfi' -t \vfi'' \right) \left( \frac{x}{\vro}\otimes {\bf e}_n +{\bf e}_n \otimes \frac{x}{\vro}\right) \right\}
\end{array}
$$
where $t=\frac{x_n}{\vro}$, and ${\bf e}_n =(0,\ldots, 0,1)$.

According to Lemma \ref{eigenvalues}, the eigenvalues of $D^2\Phi (x)$ are
\begin{itemize}

\item $\dis \vro^{\alpha-2} \left( \alpha\,  \vfi -t\vfi' \right)$, with multiplicity  $n-2$; 
\smallskip

\item $\displaystyle  \frac{\vro^{\alpha-2}}{2} \left( \alpha^2 \vfi -t \vfi' +(1-t^2) \vfi'' \pm \sqrt{ \mathcal D} \right)$,  
\end{itemize}
with ${\mathcal D}=\left ( \alpha (\alpha -2) \vfi +t \vfi' -(1-t^2) \vfi'' \right)^2 +4 (\alpha -1)^2 (1-t^2) \vfi'^2$. 

Now, if we choose $\vfi $ of the form
$$
\vfi (t)= \phi (\arccos t)\,,
$$
then we obtain
$$
-\sqrt{1-t^2} \vfi' =\phi'\,,\quad (1-t^2)\vfi'' -t\vfi' = \phi''\, ,
$$
hence the conclusion.
\end{proof}

\begin{remark}
{\rm  An analogous computation shows that for smooth symmetric functions of separable variables of the form
$$
\Phi (x) = \psi (\rho)\, \phi (\theta)\, ,
$$
if  $\psi :(0,+\infty )\to \R$ is of class  $C^2$ then  the eigenvalues of the hessian matrix $D^2 \Phi(x)$ are
\begin{itemize}

\item $\dis  \left( \frac{\psi'}{\rho} \phi +\frac{\psi}{\rho^2} \frac{\phi'}{\tan \theta} \right)$, with multiplicity  $n-2$; 
\smallskip

\item $\displaystyle  \frac{ \left( \psi'' +\frac{\psi'}{\rho} \right) \phi +\frac{\psi}{\rho^2} \phi'' \pm \sqrt{ \mathcal D}}{2}$,  
\end{itemize}
with 
$$
 {\mathcal D} =\left( \left( \psi''-\frac{\psi'}{\rho}\right) \phi -\frac{\psi}{\rho^2} \phi''\right)^2 +4 \left( \frac{\psi'}{\rho} -\frac{\psi}{\rho^2} \right)^2 \phi'^2 \, .
$$
\hfill$\Box$
}\end{remark}
\medskip

For what follows, it is convenient to introduce the following functions. Let $\omega \geq 1$ be a parameter and let $g_\omega :\, (-\infty ,\ 1-\omega )\cup (1-1/\omega , +\infty) \to (0,+\infty)$ be the positive function defined as
\begin{equation}\label{gom}
g_\omega (\alpha) = \left\{ \begin{array}{ll}
 \arctan \sqrt{\omega} +\frac{2-\alpha}{\sqrt{(\alpha -1+1/\omega)(\alpha-1+\omega)}} \arctan \sqrt{\frac{\omega (\alpha -1+1/\omega)}{\alpha-1+\omega} } &    \hbox{ if } \alpha\geq 1\\[2ex]
 \arctan \frac{1}{\sqrt{\omega}} +\frac{2-\alpha}{\sqrt{(\alpha -1+1/\omega)(\alpha-1+\omega)}} \arctan \sqrt{\frac{\alpha -1+\omega}{\omega (\alpha-1+1/\omega)}}  &    \hbox{ if } 1-\frac{1}{\omega}<\alpha<1\\[2ex]
 -\arctan \sqrt{\omega} +\frac{2-\alpha}{\sqrt{(\alpha -1+1/\omega)(\alpha-1+\omega)}} \arctan \sqrt{\frac{\omega (\alpha -1+1/\omega)}{\alpha-1+\omega}} &    \hbox{ if } \alpha < 1-\omega
\end{array}\right.
\end{equation}

For $\omega \geq 1$ and $\alpha \in (-\infty ,\ 1-\omega )\cup (1-1/\omega , +\infty)$, let further 
$G_{\omega, \alpha} : [0,1] \to [0, g_\omega(\alpha)]$ be defined as
\begin{equation}\label{Gom}
G_{\omega, \alpha} (x) = \left\{ \begin{array}{ll}
\arctan \left(\sqrt{\omega}x\right) +
\frac{2-\alpha}{\sqrt{(\alpha -1+1/\omega)(\alpha-1+\omega)}} \arctan \left( \sqrt{\frac{\omega (\alpha -1+1/\omega)}{\alpha-1+\omega}}x\right) &  \hbox{ for } \alpha\geq 1\\[2ex]
\arctan \left(\frac{x}{\sqrt{\omega}}\right) +
\frac{2-\alpha}{\sqrt{(\alpha -1+1/\omega)(\alpha-1+\omega)}} \arctan \left(\sqrt{\frac{\alpha -1+\omega}{\omega (\alpha-1+1/\omega)}}x\right) & \hbox{ for } 1-\frac{1}{\omega}<\alpha<1\\[2ex]
-\arctan\left(\sqrt{\omega}x\right) +
\frac{2-\alpha}{\sqrt{(\alpha -1+1/\omega)(\alpha-1+\omega)}} \arctan \left(\sqrt{\frac{\omega (\alpha -1+1/\omega)}{\alpha-1+\omega}}x\right) &  \hbox{ for } \alpha < 1-\omega
\end{array}\right.
\end{equation}
Note that 
$$
g_1(\alpha)= \frac{\pi}{2|\alpha|}\,,\quad G_{1,\alpha}(x)=\frac{2 \arctan x}{|\alpha|}\,,
$$
and, in general,
$$
G_{\omega,\alpha}(1)=g_\omega (\alpha)\,.
$$
The function $g_\omega$ is strictly increasing on $(-\infty, 1-\omega)$ and strictly decreasing on $(1-1/\omega ,+\infty)$, as well as $G_{\omega, \alpha}$ is strictly increasing on $[0,1]$. Moreover, $g_\omega$ strictly increases with respect to $\omega$, except for $\alpha=1$ where $g_\omega (1)=\frac{\pi}{2}$ for all $\omega\geq 1$; $G_{\omega, \alpha}$ increases with respect to $\alpha \in (-\infty, 1-\omega)$ and decreases with respect to $\alpha \in (1-1/\omega , +\infty)$,  and it always increases with respect to $\omega \geq 1$. The graphs of $g_{\omega}$ and $G_{\omega, \alpha}$ are represented in Fig. 1 and 2 respectively.

\begin{figure}
\includegraphics[height=60mm]{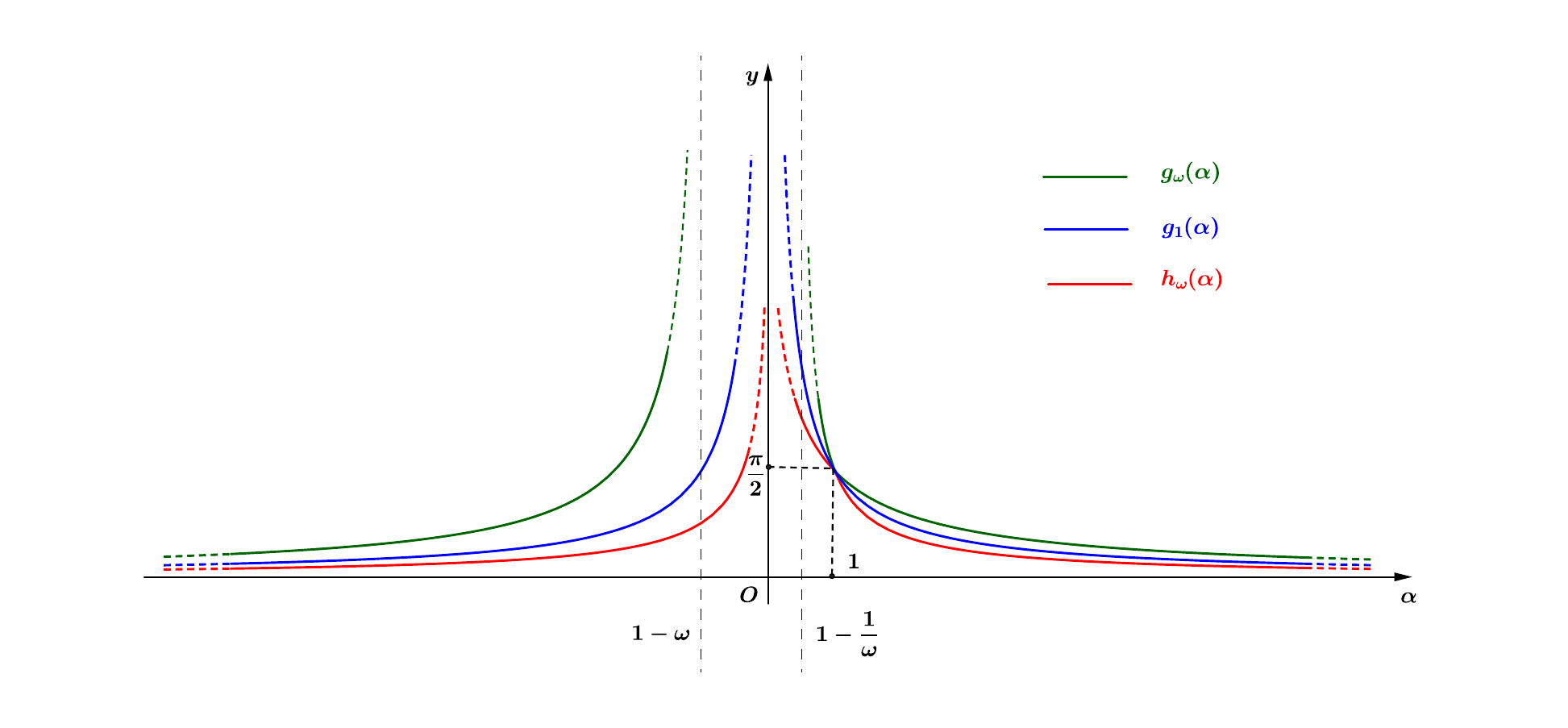}
\caption{The functions $g_\omega$ and $h_\omega$ for $\omega>1$ and $\omega=1$.}
\end{figure}

The functions $g_\omega$ and $G_{\omega, \alpha}$ will be used to determine respectively the homogeneity exponents and the angular part of the positive homogeneous solutions of problem \refe{homP}. In order to obtain the analogous quantities for the negative homogeneous solutions, let us further introduce the function $h_\omega :\, (-\infty ,0)\cup (0,+\infty) \to (0,+\infty)$  defined as
\begin{equation}\label{hom}
h_\omega (\alpha) = \left\{ \begin{array}{ll}
 \arctan \frac{1}{\sqrt{\omega}} +\frac{2-\alpha}{\sqrt{(\alpha -1+1/\omega)(\alpha-1+\omega)}} \arctan \sqrt{\frac{ \alpha -1+\omega}{\omega (\alpha-1+1/\omega)} } &    \hbox{ if } \alpha\geq 1\\[2ex]
 \arctan\sqrt{\omega} +\frac{2-\alpha}{\sqrt{(\alpha -1+1/\omega)(\alpha-1+\omega)}} \arctan \sqrt{\frac{\omega(\alpha -1+1/\omega)}{\alpha-1+\omega}}  &    \hbox{ if } 1-\frac{1}{\omega}< \alpha<1\\[2ex]
  \arctan\sqrt{\omega} +\frac{2-\alpha}{\sqrt{-(\alpha -1+1/\omega)(\alpha-1+\omega)}} {\rm arctanh} \sqrt{-\frac{\omega(\alpha -1+1/\omega)}{\alpha-1+\omega}}  &    \hbox{ if } 0<\alpha \leq 1-\frac{1}{\omega}\\[2ex]
 -\arctan \frac{1}{\sqrt{\omega}} +\frac{2-\alpha}{\sqrt{-(\alpha -1+1/\omega)(\alpha-1+\omega)}} {\rm arctanh} \sqrt{-\frac{\alpha -1+\omega}{\omega(\alpha-1+1/\omega)}} &    \hbox{ if }  1-\omega \leq \alpha <0\\[2ex]
 -\arctan \frac{1}{\sqrt{\omega}} +\frac{2-\alpha}{\sqrt{(\alpha -1+1/\omega)(\alpha-1+\omega)}} \arctan \sqrt{\frac{\alpha -1+\omega}{\omega(\alpha-1+1/\omega)}} &    \hbox{ if } \alpha < 1-\omega
\end{array}\right.
\end{equation}
and, accordingly, the function $H_{\omega, \alpha} : [0,1] \to [0,h_\omega(\alpha)]$  defined as
\begin{equation}\label{Hom}
H_{\omega, \alpha} (x) = \left\{ \begin{array}{ll}
\arctan \left( \frac{1}{\sqrt{\omega}}x\right) +\frac{2-\alpha}{\sqrt{(\alpha -1+1/\omega)(\alpha-1+\omega)}} \arctan \left(\sqrt{\frac{ \alpha -1+\omega}{\omega (\alpha-1+1/\omega)} }x\right) &    \hbox{ if } \alpha\geq 1\\[2ex]
 \arctan \left(\sqrt{\omega} x\right)+\frac{2-\alpha}{\sqrt{(\alpha -1+1/\omega)(\alpha-1+\omega)}} \arctan \left( \sqrt{\frac{\omega(\alpha -1+1/\omega)}{\alpha-1+\omega}} x\right) &    \hbox{ if } 1-\frac{1}{\omega}< \alpha<1\\[2ex]
  \arctan \left(\sqrt{\omega} x\right)+\frac{2-\alpha}{\sqrt{-(\alpha -1+1/\omega)(\alpha-1+\omega)}} {\rm arctanh} \left( \sqrt{-\frac{\omega(\alpha -1+1/\omega)}{\alpha-1+\omega}} x\right) &    \hbox{ if } 0<\alpha \leq 1-\frac{1}{\omega}\\[2ex]
 -\arctan \left( \frac{1}{\sqrt{\omega}} x\right)+\frac{2-\alpha}{\sqrt{-(\alpha -1+1/\omega)(\alpha-1+\omega)}} {\rm arctanh} \left( \sqrt{-\frac{\alpha -1+\omega}{\omega(\alpha-1+1/\omega)}}x\right) &    \hbox{ if }  1-\omega \leq \alpha <0\\[2ex]
 -\arctan \left( \frac{1}{\sqrt{\omega}}x\right) +\frac{2-\alpha}{\sqrt{(\alpha -1+1/\omega)(\alpha-1+\omega)}} \arctan \left( \sqrt{\frac{\alpha -1+\omega}{\omega(\alpha-1+1/\omega)}} x\right) &    \hbox{ if } \alpha < 1-\omega
 \end{array}\right.
\end{equation}

$h_\omega$ and $H_{\omega, \alpha}$ satisfy analogous properties as $g_\omega$ and $G_{\omega, \alpha}$ and also their graphs are pictured in Fig.1 and Fig. 2 respectively. Note that, for $\alpha \in (-\infty ,1-\omega)\cup (1-1/\omega ,+\infty)$, we formally have $h_\omega (\alpha)=g_{1/\omega}(\alpha)$ and $H_{\omega ,\alpha}(x)=G_{1/\omega,\alpha}(x)$ as well, but $h_\omega (\alpha)$ is defined also for $1-\omega \leq \alpha <0$ and $0<\alpha \leq 1-1/\omega$. Note also that, owing to the monotonicity with respect to $\omega$,  one has
$$
h_{\omega} (\alpha) < h_1(\alpha)= \frac{\pi}{2 |\alpha|} = g_1(\alpha) < g_\omega (\alpha)\quad \forall\, \alpha \neq 1\, ,\ \omega >1\, ,
$$
and, moreover,
\begin{equation}\label{1}
h_{\omega} (1) = \frac{\pi}{2} = g_\omega (1)\quad \forall\,  \omega \geq1\, .
\end{equation}

\begin{figure}
\includegraphics[height=60mm]{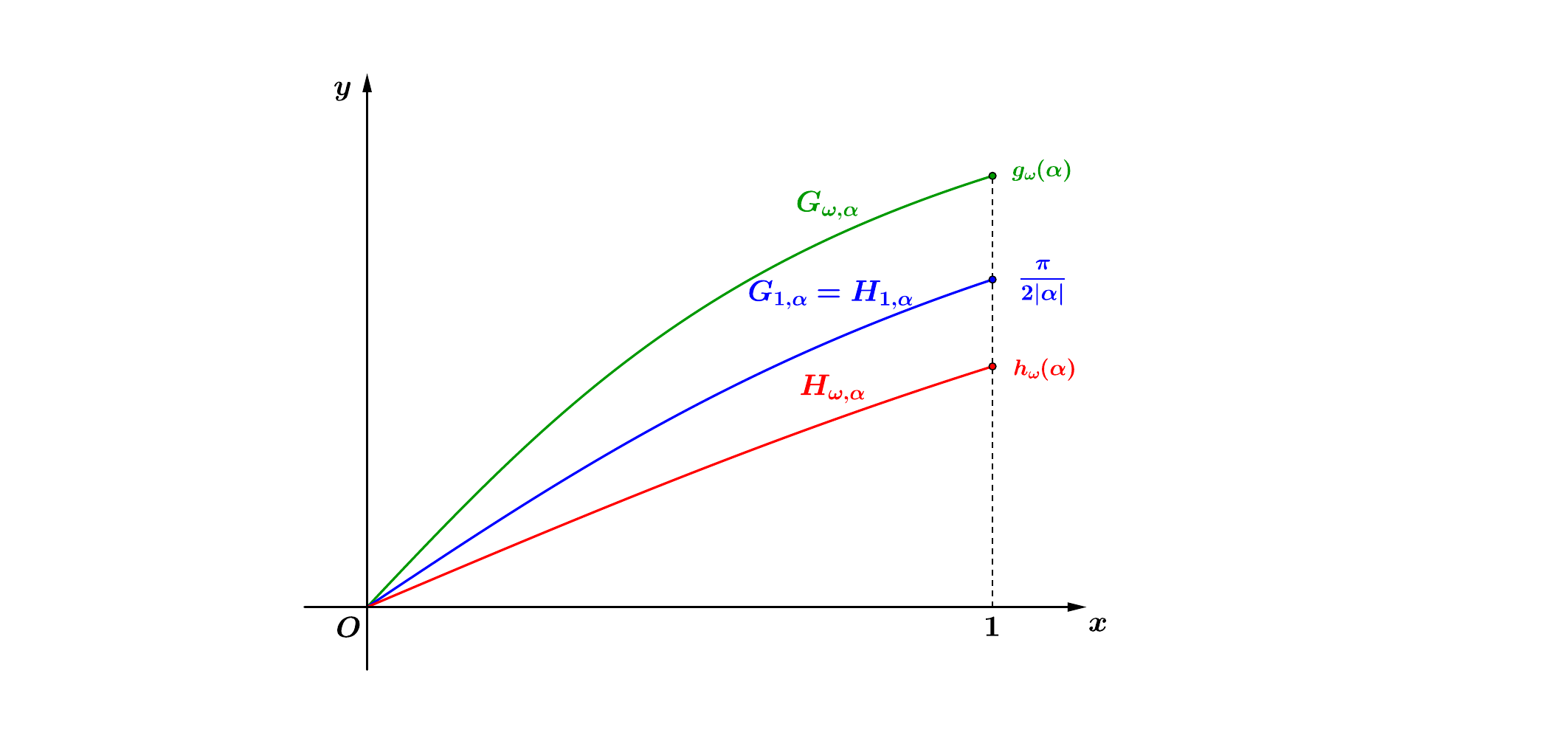}
\caption{The functions $G_{\omega, \alpha}$ and $H_{\omega, \alpha}$ for $\omega>1$ and $\omega=1$.}
\end{figure}

\begin{theorem}\label{homS}
Given $\Lambda \geq \lambda$, let $\omega=\frac{\Lambda}{\lambda}$ and let $g_\omega\, ,\ G_{\omega, \alpha}\, ,\ h_\omega$ and $H_{\omega, \alpha}$ be as in \refe{gom}, \refe{Gom}, \refe{hom} and \refe{Hom} respectively.  Then, with respect to polar coordinates $\rho=\sqrt{x^2+y^2},\ \theta= \arccos\left( \frac{y}{\rho}\right)$, the positive homogeneous solutions of problem \refe{homP} are 
$$u_{\alpha^\pm}=\rho^{\alpha^\pm} \phi_{\alpha^\pm}(\theta)$$
 where the exponents $\alpha^-<0<\alpha^+$ satisfy
\begin{equation}\label{apm}
g_\omega (\alpha^\pm)=\theta_0
\end{equation}
and the angular parts are given by
\begin{equation}\label{fi+}
\phi_{\alpha^+}(\theta)=\left\{
\begin{array}{l}
\frac{\left( 1-G_{\omega,\alpha^+}^{-1}(\theta)^2\right) \left( 1+\frac{\omega (\alpha^+-1+1/\omega)}{\alpha^+-1+\omega}G_{\omega,\alpha^+}^{-1}(\theta)^2\right)^{\frac{\alpha^+-2}{2}}}{\left( 1+\omega G_{\omega,\alpha^+}^{-1}(\theta)^2\right)^{\alpha^+/2}} \ \hbox{if } \theta_0\leq \pi/2\\[2ex]
\frac{\left( 1-G_{\omega,\alpha^+}^{-1}(\theta)^2\right) \left( 1+\frac{\alpha^+-1+\omega}{\omega(\alpha^+-1+1/\omega)}G_{\omega,\alpha^+}^{-1}(\theta)^2\right)^{\frac{\alpha^+-2}{2}}}{\left( 1+1/\omega G_{\omega,\alpha^+}^{-1}(\theta)^2\right)^{\alpha^+/2}} \ \hbox{if } \theta_0> \pi/2
\end{array}
\right.
\end{equation}
\begin{equation}\label{fi-}
\phi_{\alpha^-}(\theta)=\frac{\left( 1-G_{\omega,\alpha^-}^{-1}(\theta)^2\right) \left( 1+\frac{\omega (\alpha^--1+1/\omega)}{\alpha^--1+\omega}G_{\omega,\alpha^-}^{-1}(\theta)^2\right)^{\frac{\alpha^--2}{2}}}{\left( 1+\omega G_{\omega,\alpha^-}^{-1}(\theta)^2\right)^{\alpha^-/2}}
\end{equation}
Analogously, the negative homogeneous solutions of problem \refe{homP} are
$$v_{\beta^\pm}=-\rho^{\beta^\pm} \psi_{\beta^\pm}(\theta)$$
 where the exponents $\beta^-<0<\beta^+$ satisfy
\begin{equation}\label{bpm}
h_\omega (\beta^\pm)=\theta_0
\end{equation}
and the angular parts are given by
\begin{equation}\label{psi+}
\psi_{\beta^+}(\theta)=\left\{
\begin{array}{l}
\frac{\left( 1-H_{\omega,\beta^+}^{-1}(\theta)^2\right) \left( 1+\frac{\beta^+-1+\omega}{\omega (\beta^+-1+1/\omega)}H_{\omega,\beta^+}^{-1}(\theta)^2\right)^{\frac{\beta^+-2}{2}}}{\left( 1+1/\omega H_{\omega,\beta^+}^{-1}(\theta)^2\right)^{\beta^+/2}} \ \hbox{if } \theta_0\leq \pi/2\\[2ex]
\frac{\left( 1-H_{\omega,\beta^+}^{-1}(\theta)^2\right) \left( 1+\frac{\omega(\beta^+-1+1/\omega)}{\beta^+-1+1/\omega}H_{\omega,\beta^+}^{-1}(\theta)^2\right)^{\frac{\beta^+-2}{2}}}{\left( 1+\omega H_{\omega,\beta^+}^{-1}(\theta)^2\right)^{\beta^+/2}} \ \hbox{if } \theta_0> \pi/2
\end{array}
\right.
\end{equation}
\begin{equation}\label{psi-}
\psi_{\beta^-}(\theta)=\frac{\left( 1-H_{\omega,\beta^-}^{-1}(\theta)^2\right) \left( 1+\frac{ \beta^--1+\omega}{\omega (\beta^--1+1/\omega)}H_{\omega,\beta^-}^{-1}(\theta)^2\right)^{\frac{\beta^--2}{2}}}{\left( 1+1/\omega H_{\omega,\beta^-}^{-1}(\theta)^2\right)^{\beta^-/2}}
\end{equation}
\end{theorem}

\begin{proof} By the definitions \refe{apm} and \refe{bpm} of the homogeneity exponents it follows that 
$$
\left( G_{\omega ,\alpha^{\pm}}\right)^{-1} (\theta_0)=\left( H_{\omega ,\beta^{\pm}}\right)^{-1} (\theta_0)=1\, .
$$
Therefore, if $\phi_{\alpha^{\pm}}$ are defined as in  \refe{fi+}, \refe{fi-}, and $\psi_{\beta^\pm}$ are as in \refe{psi+}, \refe{psi-}, the boundary condition in problem \refe{homP} is satisfied both by $u_{\alpha^\pm}$ and $v_{\beta^\pm}$.

Let us now assume  $\theta_0 \leq \pi/2$. By \refe{1} and the monotonicity of $g_\omega$, we then have $\alpha^+\geq 1$. In order to check that $u_{\alpha^+}=\rho^{\alpha^+}\phi_{\alpha^+} (\theta)$ is a solution of \refe{homP}, we make use of  Lemma \ref{hessian}. By straightforward computation, we obtain
$$
\begin{array}{l}
\phi_{\alpha^+}' =- \frac{\alpha^+ (\omega+1)}{\sqrt{\omega}} \frac{ \left( 1+\frac{\omega (\alpha^+-1+1/\omega)}{\alpha^+-1+\omega}\left( G_{\omega,\alpha^+}^{-1}\right)^2\right)^{\frac{\alpha^+-2}{2}}}{\left( 1+\omega  \left(G_{\omega,\alpha^+}^{-1}\right)^2\right)^{\alpha^+/2}} G_{\omega,\alpha^+}^{-1}\\[2ex]
\phi_{\alpha^+}'' =- \frac{\alpha^+ (\alpha^+-1+\omega)}{\omega} \frac{ \left( 1+\frac{\omega (\alpha^+-1+1/\omega)}{\alpha^+-1+\omega}\left( G_{\omega,\alpha^+}^{-1}\right)^2\right)^{\frac{\alpha^+-2}{2}}}{\left( 1+\omega  \left(G_{\omega,\alpha^+}^{-1}\right)^2\right)^{\alpha^+/2}}  \left(
1- \frac{\omega^2 (\alpha^+-1+1/\omega)}{\alpha^+-1+\omega} \left( G_{\omega,\alpha^+}^{-1}\right)^2\right)
\end{array}
$$
According to Lemma \ref{hessian}, after some calculations, we find out that the eigenvalues of the  hessian matrix $D^2u_{\alpha^+}$ are 
$$
\lambda_1= \alpha^+ \omega \left( \alpha^+-1\right) \rho^{\alpha^+-2} \frac{ \left( 1+\frac{\omega (\alpha^+-1+1/\omega)}{\alpha^+-1+\omega}\left( G_{\omega,\alpha^+}^{-1}\right)^2\right)^{\frac{\alpha^+-2}{2}}}{\left( 1+\omega  \left(G_{\omega,\alpha^+}^{-1}\right)^2\right)^{\alpha^+/2}} \left( \frac{1}{\omega}+ \left( G_{\omega, \alpha^+}^{-1}\right)^2\right)>0
$$
and 
$$
\lambda_2 =-\frac{\lambda_1}{\omega}<0\, .
$$
Therefore, we obtain
$$
\Mm \left(D^2u_{\alpha^+}\right) = \lambda\, \lambda_1 +\Lambda\, \lambda_2= \lambda \left( \lambda_1+\omega \lambda_2\right)=0\, .
$$
The other cases can be carried out analogously, and the proof is completed.

\end{proof}
\smallskip

\begin{remark}
{\rm The ODE problem solved by the pairs $(\alpha^\pm , \phi_{\alpha^\pm})$ is the nonlinear eigenvalue problem
$$
\left\{ \begin{array}{c}
\dis \phi_\alpha'' +\alpha \left( \alpha +\frac{\gamma}{2} (\alpha -1)\right) \phi_\alpha = |\alpha -1|\, \sqrt{\gamma} \sqrt{\alpha^2 \left( 1+\frac{\gamma}{4}\right) \phi_\alpha^2 +\phi_\alpha'^2} \qquad \hbox{for } |\theta|<\theta_0\\[2ex]
\phi_\alpha (\pm \theta_0)=0
\end{array} \right.
$$
with $\gamma= \frac{(\omega-1)^2}{\omega}$, which is the nonlinear extension, when  $\omega>1$, of the linear case
$$
\left\{ \begin{array}{c}
\dis \phi_\alpha'' +\alpha^2  \phi_\alpha =0 \qquad \hbox{for } |\theta|<\theta_0\\[2ex]
\phi_\alpha (\pm \theta_0)=0
\end{array} \right.
$$
 obtained for $\omega=1$, i.e.  $\gamma=0$. The explicit expressions \refe{apm}, \refe{fi+} and \refe{fi-} have been obtained by looking for solutions $\phi_\alpha$ such that $\phi_\alpha'=\zeta (\phi_\alpha)$ for some smooth function $\zeta$. This ansatz leads to an easily integrable  first order ODE for the unknown function $\zeta$. 
 
 An analogous problem can be obtained when looking for homogeneous solutions in a symmetric cone of the $n$-dimensional euclidean space, with $n\geq 3$. But, in this case,  Lemma \ref{hessian} implies that one obtains a difficult to integrate non autonomous fully nonlinear ODE. 
}
\end{remark}
\begin{remark} 
{\rm  Let us remark that, for $\theta_0=\pi/2$, i.e. if $\mathcal{C}_0$ is the upper halfplane, then $\pm y$ are clearly the positive and negative homogeneous solutions with positive homogeneity exponent. Consistently, from the above theorem  and \refe{1}, we recover  in this case $\alpha^+=\beta^+=1$, by \refe{Gom} and \refe{Hom} we get $G_{\omega, 1} (x)=H_{\omega,1}(x)= \arctan \left( \frac{(\omega+1) x}{\sqrt{\omega}(1-x^2)}\right)$ and \refe{fi+} and \refe{psi+} yield $\phi_1(\theta)=\psi_1(\theta)= \cos \theta$.

 Other simple cases occur either if $\theta_0=\arctan \sqrt{\omega}$ or if $\theta_0 =\arctan  \frac{1}{\sqrt{\omega}}$. Indeed, according to Theorem \ref{homS}, if $\theta_0=\arctan \sqrt{\omega}$, then $\alpha^+=2$, and then $G_{\omega, 2}(x)=\arctan (\sqrt{\omega}x)$, $\phi_2(\theta)=\cos^2 (\theta) -\frac{1}{\omega} \sin^2 (\theta)$ and $u_2(x,y)=y^2-\frac{x^2}{\omega}$.
Symmetrically, if $\theta_0=\arctan \frac{1}{\sqrt{\omega}}$, one has $\beta^+=2$ and $v_2(x,y)=\omega x^2-y^2$. This observation inspired the explicit construction in \cite{BL} of the positive principal  eigenfunction for the operator $\Mp$ in special planar domains with corners of amplitude $2\arctan \frac{1}{\sqrt{\omega}}$.
}
\end{remark}

As a first consequence of Theorem \ref{homS} we can deduce some relationships between the homogeneity exponents $\alpha^\pm$ and $\beta^\pm$.

\begin{corollary}\label{rela}
 Given $\Lambda \geq \lambda$ and $\theta_0\in (0, \pi)$ let $\alpha^-(\theta_0)<0<\alpha^+(\theta_0)$ and $\beta^-(\theta_0)<0<\beta^+(\theta_0)$ be the homogeneity exponents of the respectively positive and negative  homogeneous solutions of problem \refe{homP}. Then, for any $\theta_0\in \left( 0, \frac{\pi}{2}\right)$ one has
$$
\begin{array}{c}
\alpha^+ \left( \theta_0+\frac{\pi}{2}\right) = \frac{\alpha^-(\theta_0)}{\alpha^-(\theta_0)-1}\\[2ex]
\beta^+ \left( \theta_0+\frac{\pi}{2}\right) = \frac{\beta^-(\theta_0)}{\beta^-(\theta_0)-1}\\[2ex]
\beta^+ \left( \frac{\pi}{2} -\theta_0\right) = \frac{\alpha^+(\theta_0)}{\alpha^+(\theta_0)-1}
\end{array}
$$
\end{corollary}

\begin{proof} Let $\omega=\Ll$ and $g_\omega$, $h_\omega$ be as in \refe{gom} and \refe{hom}. Then, by the right definition, we obtain the following identities:
$$
\begin{array}{l}
g_\omega (\alpha)-g_\omega \left( \frac{\alpha}{\alpha-1}\right)=\frac{\pi}{2} \qquad \hbox{ for all }\ 1-\frac{1}{\omega} <\alpha <1\\[2ex]
h_\omega (\alpha)-h_\omega \left( \frac{\alpha}{\alpha-1}\right)=\frac{\pi}{2} \qquad  \hbox{ for all }\  0 <\alpha <1\\[2ex]
h_\omega (\alpha)+g_\omega \left( \frac{\alpha}{\alpha-1}\right)=\frac{\pi}{2} \qquad  \hbox{ for all }\  \alpha \geq 1\, .
\end{array}
$$
The conclusion then  follows from the characterizations  \refe{apm} and \refe{bpm} of $\alpha^\pm$ and $\beta^\pm$ given in Theorem \ref{homS}.

\end{proof}
\smallskip

\begin{remark}
{\rm In the paper \cite{L} an estimate for any dimension $n\geq 2$ of $\alpha^-\left( \frac{\pi}{2}\right)$was given, namely
\begin{equation}\label{leoni}
1-n\, \omega\leq \alpha^-\left( \frac{\pi}{2}\right)\leq -\omega (n-1)\, .
\end{equation}
We observe that,  by Theorem \ref{homS}, for $n=2$ inequalities \refe{leoni} amount to
$$
\left\{
\begin{array}{l}
\dis-\arctan \sqrt{\omega} +\frac{2 \omega+1}{\sqrt{2\omega^2-1}} \arctan\sqrt{\frac{2\omega^2-1}{\omega}}\leq \frac{\pi}{2}\\[2ex]
\dis -\arctan \sqrt{\omega}+ \frac{\sqrt{\omega}(2+\omega)}{\sqrt{\omega^2+\omega-1}}\arctan \sqrt{\omega^2+\omega-1} \geq \frac{\pi}{2}
\end{array}
\right.
$$
One can check that they indeed hold true and equalities occur only for $\omega=1$. 
However, the bounds \refe{leoni} for $n=2$ and Corollary \ref{rela} yield the global lower bound 
$$
\alpha^+(\theta) \geq \frac{\omega}{\omega+1}\qquad \hbox{ for all }\ \theta \in (0,\pi)\, .
$$
}
\end{remark}

\section{Some applications}

\subsection{Monotonicity formulas for supersolutions} 

As an immediate consequence of the comparison principle and the explicit knowledge of the homogeneous solutions $u_{\alpha^\pm}$, one can obtain bounds on nonnegative  supersolutions of  the extremal equations. In particular, we have the following monotonicity statements.
\begin{theorem}\label{monotonicity}
Let $u:\overline{\mathcal{C}_0}\to [0,+\infty]$  be a lower semicontinuous solution of  $\Mm (D^2u)\leq 0$ in $\mathcal{C}_0$, and, for $r>0$,
let us define
$$
m^\pm(r)\,:= \inf_{\partial B_r\cap \mathcal{C}_0} \frac{u}{\phi_{\alpha^\pm}}
$$
Then
\begin{equation}\label{monom+}
r\in (0,+\infty)\mapsto r^{-\alpha^+} m^+ (r)\ \hbox {is non increasing }
\end{equation}
\begin{equation}\label{monom-}
r\in (0,+\infty)\mapsto r^{-\alpha^-} m^- (r)\ \hbox{is non decreasing}\, .
\end{equation}
\end{theorem}
\begin{proof}
Let $r>0$. The comparison principle in the domain $\mathcal{C}_0\cap B_r$ yields
$$
u\geq \left( \inf_{\partial B_r\cap \mathcal{C}_0}\frac{u}{u_{\alpha^+}}\right) \, u_{\alpha^+}=\frac{m^+(r)}{r^{\alpha^+}}u_{\alpha^+}\quad \hbox{in } \mathcal{C}_0\cap B_r\, ,
$$
hence 
$$
\rho^{-\alpha^+}m^+(\rho)\geq r^{-\alpha^+}m^+(r)\quad \hbox{ for all } \rho\leq r\, .
$$
In order to prove \refe{monom-}, let  $\epsilon >0$ be fixed. Then, there exists $R_\epsilon>0$ such that
$$
u_{\alpha^-}<\epsilon \quad \hbox{ in } \mathcal{C}_0\setminus B_{R_\epsilon}\, .
$$
For $0<r<R_\epsilon<R$, the comparison principle applied in the domain $\mathcal{C}_0\cap \left( B_R\setminus B_r\right)$ yields
$$
u+\epsilon\,  \left( \inf_{\partial B_r\cap \mathcal{C}_0}\frac{u}{u_{\alpha^-}}\right)\geq \left( \inf_{\partial B_r\cap \mathcal{C}_0}\frac{u}{u_{\alpha^-}}\right) \, u_{\alpha^-}= \frac{m^-(r)}{r^{\alpha^-}}u_{\alpha^-}\quad \hbox{in } \mathcal{C}_0\cap \left( B_R\setminus B_r\right)\, .
$$
By letting first $R\to +\infty$ and then $\epsilon\to 0$, we obtain
$$
u\geq \frac{m^-(r)}{r^{\alpha^-}}u_{\alpha^-}\quad \hbox{in } \mathcal{C}_0 \setminus B_r\, ,
$$
hence
$$
\rho^{-\alpha^-}m^-(\rho)\geq r^{-\alpha^-}m^-(r)\quad \hbox{ for all } \rho\geq r\, .
$$
\end{proof}

\subsection{Optimal non--existence Liouville type theorems} 
As it is well known for semilinear elliptic equations, it has been recently proved also for fully nonlinear elliptic equations that the orders of homogeneity 
 of the positive homogeneous solutions in cones determine the critical exponents in non existence Liouville type theorems for positive solutions of differential inequalities having power like zero order terms, see \cite{AS} and \cite{ L} for the case of halfspaces. In particular, it has been proved in \cite{AS} that the inequality
\begin{equation}\label{Pp0}
 \Mm (D^2u)+u^p\leq 0\ \hbox{in } \mathcal{C}_0\ 
\end{equation}
in any cone--like domain ${\mathcal C}_0\subset \R^n$ has no positive solution if
\begin{equation}\label{condp}
1-\frac{2}{\alpha^+}\leq p\leq 1-\frac{2}{\alpha^-}\, .
\end{equation}
We remark that the statement given in \cite{AS} actually is for Laplace operator, but the proof presented there, relying only on the maximum principle and on the existence of homogeneous solutions for the homogeneous equation,  can be transposed word by word to viscosity solutions of \refe{Pp0}.

We further observe that, if $n=2$, by Theorem \ref{homS} condition \refe{condp} may be written as
$
g_\omega \left( \frac{2}{1-p}\right) \leq \theta_0\, .
$
Moreover, thanks to the explicit expressions \refe{fi+} and \refe{fi-} of the functions $\phi_{\alpha^\pm}$, we can check that the above condition is optimal for the non existence of positive solutions of \refe{Pp0}, thus obtaining the following statement.

\begin{theorem}\label{Liou}
For $\Lambda\geq \lambda>0$, $\theta_0\in (0,\pi)$ and  $p\in \R$, the inequality \refe{Pp0}
has no positive lower semicontinuous viscosity solution if and only if
$$
g_\omega \left( \frac{2}{1-p}\right) \leq \theta_0\, ,
$$
with $\omega=\frac{\Lambda}{\lambda}\geq 1$ and $g_\omega$ as in {\rm \refe{gom}}.
\end{theorem}

\begin{proof} By the discussion above,  we need only to check the existence of supersolutions of \refe{Pp0} both for $p>1-\frac{2}{\alpha^-}$ and for $p<1-\frac{2}{\alpha^+}$.

For $p>1-\frac{2}{\alpha^-}>1$, we can select  $\alpha$ and $\beta$ satisfying $\alpha^-< \alpha <\beta< \min \{1-\omega\, , -\frac{2}{p-1}\}<0$, and consider the function
\begin{equation}\label{u}
u(\rho, \theta)=\rho^\beta \left( \phi_\alpha (\theta)-\gamma\right)\, ,
\end{equation}
where $\phi_\alpha$ is defined as in \refe{fi-} with $\alpha^-$ replaced by $\alpha$, and $\gamma$ is such that $0<\gamma< \phi_\alpha(\theta_0)$.  Note that $\phi_\alpha(\theta)$ is strictly positive for $|\theta|< g_\omega(\alpha)$ and, by monotonicity, $g_\omega(\alpha)> g_\omega(\alpha^-)=\theta_0$.

By positive homogeneity and superadditivity of operator $\Mm$, we have
$$
-\Mm (D^2u)\geq -\Mm \left( D^2(\rho^\beta \phi_\alpha(\theta))\right) +\gamma \, \Mm(D^2\, \rho^\beta)\, ,
$$
so that, by a simple computation, we obtain
\begin{equation}\label{est1}
-\Mm(D^2u)\geq -\Mm \left( D^2(\rho^\beta \phi_\alpha(\theta))\right) +\lambda \, \gamma\, \beta (\beta-1+\omega) \rho^{\beta-2}\, .
\end{equation}
Moreover, according to Lemma \ref{hessian} and to the definition of $\phi_\alpha$,  the eigenvalues of $D^2(\rho^\beta \phi_\alpha(\theta))$ are given by
$$
\lambda_{1,2} = \dis \frac{\rho^{\beta-2}}{2} {\mathcal R} \left[ \beta^2  -\frac{\alpha}{\omega} (\alpha-1+\omega) -\left( \beta^2 -\omega \alpha \left(\alpha-1+\frac{1}{\omega}\right) \right)  \left(G^{-1}_{\omega,\alpha}(\theta)\right)^2 \pm \sqrt{{\mathcal D}}\right]\, ,
$$
with
$$
{\mathcal R}=\frac{ \left( 1+\frac{\omega (\alpha -1+1/\omega)}{\alpha-1+\omega} \left(G^{-1}_{\omega,\alpha}(\theta)\right)^2\right)^{(\alpha-2)/2}}{\left( 1+\omega  \left(G^{-1}_{\omega,\alpha}(\theta)\right)^2\right)^{\alpha/2}}
$$
and
$$
\begin{array}{ll}
{\mathcal D} = &
\left[  (\beta-1)^2  +(\alpha-1)\, \left(1 +\frac{\alpha}{\omega}\right) -\left( (\beta-1)^2  +(\alpha-1)\, (1 +\omega \alpha ) \right)  \left(G^{-1}_{\omega,\alpha }(\theta)\right)^2\right]^2 \\[2ex]
& +4 \, (\beta-1)^2  \frac{(\omega+1)^2}{\omega} \alpha^2 \left(G^{-1}_{\omega,\alpha}(\theta)\right)^2\, .
\end{array}
$$
Since $\alpha<\beta<1-\omega<0$, by choosing $\beta$ sufficiently close to $\alpha$ it is not difficult to verify that
$$
\begin{array}{rl}
\dis {\mathcal D}   \geq &
\dis \left[  (\beta-1)^2  +(\alpha-1)\, \left(1 +\frac{\alpha}{\omega}\right) +\left( (\beta-1)^2  +(\alpha-1)\, (1 +\omega \alpha ) \right)  \left(G^{-1}_{\omega,\alpha }(\theta)\right)^2\right]^2\\[2ex]
 = & \dis  \left[ \beta (\beta-2) +\frac{\alpha}{\omega} (\alpha -1+\omega) + \left( \beta(\beta-2) +\omega \alpha \left( \alpha-1+\frac{1}{\omega}\right) \right) \left(G^{-1}_{\omega,\alpha }(\theta)\right)^2\right]^2\, .
 \end{array}
$$
Hence, one has $\lambda_1\geq 0$, $\lambda_2\leq 0$ and
$$
\begin{array}{l}
\dis -\Mm \left(D^2(\rho^\beta \phi_\alpha(\theta)\right) \geq  \lambda\, 
\rho^{\beta-2} {\mathcal R}\\[3ex]
 \dis \ \  \times \left[ \alpha (\alpha -1+\omega) -\beta (\beta-1+\omega) -\omega \left( \alpha \left( \alpha-1+\frac{1}{\omega}\right) -\beta
\left(\beta -1+\frac{1}{\omega}\right) \right)  \left(G^{-1}_{\omega,\alpha}(\theta)\right)^2\right] 
\end{array}
$$
We further observe that
$$
{\mathcal R}_0 \leq {\mathcal R}\leq 1
$$
for a positive constant ${\mathcal R}_0$ depending only on $\omega$ and $\alpha$. Therefore, we have the estimate
$$
\begin{array}{rl}
\dis -\Mm \left(D^2(\rho^\beta \phi_\alpha(\theta)\right) \geq & \dis \lambda\, 
\rho^{\beta-2} {\mathcal R}_0 \left[ \alpha (\alpha -1+\omega) -\beta (\beta-1+\omega) \right]\\[3ex]
 & \dis  -\omega \, \lambda\, 
\rho^{\beta-2} \left[ \alpha \left( \alpha-1+\frac{1}{\omega}\right) -\beta
\left(\beta -1+\frac{1}{\omega}\right) \right]  \left(G^{-1}_{\omega,\alpha}(\theta)\right)^2
\end{array}
$$
which, plugged into \refe{est1}, yields
$$
\begin{array}{rl}
\dis -\Mm \left(D^2u\right) \geq & \dis \lambda\, 
\rho^{\beta-2} \left\{  \phantom{\frac{1}{2}}\!\!\! {\mathcal R}_0\left[  \alpha (\alpha -1+\omega) -\beta (\beta-1+\omega) \right] +\gamma\, \beta\, (\beta-1+\omega) \right. \\[3ex]
 & \dis  \qquad \quad \left. -\omega  \, \left[ \alpha \left( \alpha-1+\frac{1}{\omega}\right) -\beta
\left(\beta -1+\frac{1}{\omega}\right) \right]  \left(G^{-1}_{\omega,\alpha}(\theta)\right)^2\right\} 
\end{array}
$$
Next,  we choose the constant $\gamma$ satisfying
$$
\gamma > \frac{ \omega \left[ \alpha \left( \alpha-1+\frac{1}{\omega}\right) -\beta  
\left(\beta -1+\frac{1}{\omega}\right) \right]-{\mathcal R}_0 \left[   \alpha (\alpha -1+\omega) -\beta (\beta-1+\omega)\right]}{\beta (\beta-1+\omega)}\, .
$$
Note that the right hand side of the above inequality tends to zero as $\beta\to \alpha$, so that this condition is compatible with the initial requirement $\gamma< \phi_\alpha(\theta_0)$ for $\beta$ sufficiently close to $\alpha$.

With these choices of $\alpha ,\beta$ and $\gamma$, it then follows, for $|\theta|<\theta_0$,
$$
\begin{array}{rl}
\dis -\Mm \left(D^2u\right) \geq & \dis \lambda\, 
 \omega\, \left[  \alpha \left(\alpha -1+\frac{1}{\omega}\right) -\beta \left(\beta-1+\frac{1}{\omega}\right) \right] 
 \left( 1- \left(G^{-1}_{\omega,\alpha}(\theta)\right)^2\right) \rho^{\beta-2} \\[3ex]
 \geq & \dis 
 \lambda\, 
 \omega\, \left[  \alpha \left(\alpha -1+\frac{1}{\omega}\right) -\beta \left(\beta-1+\frac{1}{\omega}\right) \right] 
 \left( 1- \left(G^{-1}_{\omega,\alpha}(\theta_0)\right)^2\right) \rho^{\beta-2}
 \end{array}
$$
and, since $\beta -2\geq \beta\, p$ and $\phi_\alpha-\gamma<\phi_\alpha\leq 1$, this implies, for $\rho\geq 1$,
$$
-\Mm (D^2 u)\geq \delta\, \rho^{\beta p}\geq \delta\, u^p\, ,
$$
where we have set
$$
\delta= \lambda\, 
 \omega\, \left[  \alpha \left(\alpha -1+\frac{1}{\omega}\right) -\beta \left(\beta-1+\frac{1}{\omega}\right) \right] 
 \left( 1- \left(G^{-1}_{\omega,\alpha}(\theta_0)\right)^2\right)\, .
 $$
 Hence, the function $\tilde{u} (x,y)=\delta^{1/(p-1)} u(x,y+1/\sin (\theta_0))$ is a classical solution of \refe{Pp0}.
 \smallskip
 
 If $p<1-\frac{2}{\alpha^+}<1$, we can apply an analogous argument as above, and even simpler in the case $\alpha^+\leq 1$, i.e. $\theta_0\geq \frac{\pi}{2}$. By considering again the function $u$ defined as in \refe{u}, but with parameters satisfying $\max\left\{ 1-\frac{1}{\omega}, \frac{2}{1-p}\right\} <\beta <\alpha <\alpha^+$ and $0<\gamma<\phi_{\alpha}(\theta_0)$, and by using again Lemma \ref{hessian} and the expressions \refe{fi+} for $\phi_\alpha$, one verifies that, for $|\theta|\leq \theta_0$, $\beta<\alpha$ sufficiently close to $\alpha$ and $\gamma$ suitably chosen, $u$ satisfies in the classical sense
 $$
 -\Mm (D^2u)\geq \delta\, \rho^{\beta-2}\, ,
 $$
 with $\delta>0$ defined exactly as before. Observing further that $\beta-2> p \beta$ and $\phi_\alpha(\theta_0)-\gamma \leq \phi_\alpha(\theta)-\gamma<1$ for $|\theta|\leq \theta_0$, we obtain for $\rho\geq 1$ and $|\theta|\leq \theta_0$
 $$
 -\Mm(D^2u)\geq \hat{\delta} \, u^p\, ,
 $$
 with
 $$
 \hat{\delta}=\left\{
 \begin{array}{ll}
 \delta & \hbox{ if } p\geq 0\\[2ex]
 \delta \left( \phi_\alpha(\theta_0)-\gamma\right)^{-p}  & \hbox{ if } p< 0
 \end{array}\right.
 $$
 Thus,   $\tilde{u} (x,y)=\hat{\delta}^{1/(p-1)} u(x,y+1/\sin (\theta_0))$  is a classical solution of \refe{Pp0}.
 
\end{proof}
\smallskip

{\bf Acknowledgments. } The author is partially supported by Gruppo Nazionale per l'Analisi Matematica e la Probabilit\`a (GNAMPA) of the Istituto Nazionale di Alta Matematica (INdAM).

\end{document}